\documentclass{amsart}
\usepackage[margin=1in]{geometry}
\usepackage{graphicx} 
\usepackage{mathtools}
\usepackage{amsthm,amssymb}
\usepackage[usenames,dvipsnames]{xcolor}
\usepackage{tikz,hyperref}
\usetikzlibrary{cd}
\usepackage{enumerate}
\usepackage{comment}

\theoremstyle{plain}
\newtheorem{thm}{Theorem}[section]
\newtheorem{prop}[thm]{Proposition}
\newtheorem{lem}[thm]{Lemma}

\newtheorem{cor}[thm]{Corollary}

\theoremstyle{remark}
\newtheorem{rmk}[thm]{Remark}

\theoremstyle{definition}
\newtheorem{definition}[thm]{Definition}
\newtheorem{example}[thm]{Example}

\newtheorem{notation}[thm]{Notation}

\newcommand{\bC}{\mathbb{C}}

\newcommand{\bA}{\mathbb{A}}

\newcommand{\asc}{\mathrm{asc}}

\newcommand{\la}{\lambda}

\newcommand{\Exp}{\mathrm{Exp}}
\newcommand{\Ht}{\widetilde{\mathrm{H}}}
\newcommand{\Add}{{\mathrm{Add}}}
\newcommand{\SYT}{{\mathrm{SYT}}}
\newcommand{\shape}{{\mathrm{sh}}}
\newcommand{\weight}{\mathrm{weight}}

\newcommand{\e}{\mathbf{e}}

\newcommand{\low}{\mathrm{m}}
\newcommand{\rowdiff}{\mathrm{d}}

\renewcommand{\L}{\mathrm{L}}
\newcommand{\crap}{\heartsuit}

\newcommand\partitionfr[1]{
	\coordinate (prev) at (0,0);
	\foreach \dir in {#1}{
		\draw[help lines, line width = .25mm] (prev) -- +(0,1) coordinate (prev);
		\draw[help lines, line width = .25mm] (prev)+(0,-1) grid +(\dir,0);
	};
}

\newcommand\dyckpathshade[3]{
	start point, size, Dyck word (size x 2 booleans)
	\fill[blue]  (#1) rectangle +(#2,#2);
	\fill[white]
	(#1)
	\foreach \dir in {#3}{
		\ifnum\dir=0
		-- ++(1,0)
		\else
		-- ++(0,1)
		\fi
	} |- (#1);
}
\newcommand\dyckpath[3]{

	\draw[help lines, gray!60, thin] (#1) grid +(#2,#2);
	\draw[dashed] (#1) -- +(#2,#2);
	\coordinate (prev) at (#1);
	\foreach \dir in {#3}{
		\ifnum\dir=0
		\coordinate (dep) at (1,0);
		\else
		\coordinate (dep) at (0,1);
		\fi
		\draw[line width=3pt,] (prev) -- ++(dep) coordinate (prev);
	};
}
\newcommand\dyckpathmn[4]{
	\draw[gray!60, thin] (#1) grid +(#2,#3);
	\coordinate (prev) at (#1);
	\foreach \dir in {#4}{
		\ifnum\dir=0
		\coordinate (dep) at (1,0);
		\else
		\coordinate (dep) at (0,1);
		\fi
		\draw[line width=3pt,black] (prev) -- ++(dep) coordinate (prev);
	};
}
\newcommand\dyckpathmnnogrid[4]{
	\coordinate (prev) at (#1);
	\foreach \dir in {#4}{
		\ifnum\dir=0
		\coordinate (dep) at (1,0);
		\else
		\coordinate (dep) at (0,1);
		\fi
		\draw[line width=3pt,black] (prev) -- ++(dep) coordinate (prev);
	};
}

\makeatletter
\renewcommand*{\eqref}[1]{%
  \hyperref[{#1}]{\textup{\tagform@{\ref*{#1}}}}%
}
\makeatother

\makeatletter

\pgfkeys{
	/tikz/sharp angle/.code={%
		\pgfsetarrowoptions{sharp >}{#1}%
		\pgfsetarrowoptions{sharp <}{-#1}%
	},
	/tikz/sharp > angle/.code={%
		\pgfsetarrowoptions{sharp >}{#1}%
	},
	/tikz/sharp < angle/.code={%
		\pgfsetarrowoptions{sharp <}{#1}%
	},
	/tikz/sharp protrude/.code=\csname if#1\endcsname\qrr@tikz@sharp@z@-0.05\p@\else\qrr@tikz@sharp@z@\z@\fi,
	/tikz/sharp protrude/.default=true
}

\newdimen\qrr@tikz@sharp@z@
\qrr@tikz@sharp@z@\z@
\pgfarrowsdeclare{sharp >}{sharp >}{%
	\edef\pgf@marshal{\noexpand\pgfutil@in@{and}{\pgfgetarrowoptions{sharp >}}}%
	\pgf@marshal
	\ifpgfutil@in@
	\edef\pgf@tempa{\pgfgetarrowoptions{sharp >}}
	\expandafter\qrr@tikz@sharp@parse\pgf@tempa\@qrr@tikz@sharp@parse
	\else
	\qrr@tikz@sharp@parse\pgfgetarrowoptions{sharp >}and-\pgfgetarrowoptions{sharp >}\@qrr@tikz@sharp@parse
	\fi
	\pgfmathparse{max(\pgf@tempa,\pgf@tempb,0)}%
	\let\qrr@tikz@sharp@max\pgfmathresult
	\pgfmathsetlength\pgf@xa{.5*\pgflinewidth * tan(\qrr@tikz@sharp@max)}%
	\pgfarrowsleftextend{+\pgf@xa}%
	\pgfarrowsrightextend{+\pgf@xa}%
}{%
	\edef\pgf@marshal{\noexpand\pgfutil@in@{and}{\pgfgetarrowoptions{sharp >}}}%
	\pgf@marshal
	\ifpgfutil@in@
	\edef\pgf@tempa{\pgfgetarrowoptions{sharp >}}
	\expandafter\qrr@tikz@sharp@parse\pgf@tempa\@qrr@tikz@sharp@parse
	\else
	\qrr@tikz@sharp@parse\pgfgetarrowoptions{sharp >}and-\pgfgetarrowoptions{sharp >}\@qrr@tikz@sharp@parse
	\fi
	\pgfmathsetlength\pgf@ya{.5*\pgflinewidth * tan(max(\pgf@tempa,\pgf@tempb,0))}%
	\pgfmathsetlength\pgf@xa{-.5*\pgflinewidth * tan(\pgf@tempa)}%
	\pgfmathsetlength\pgf@xb{-.5*\pgflinewidth * tan(\pgf@tempb)}%
	\advance\pgf@xa\pgf@ya
	\advance\pgf@xb\pgf@ya
	\ifdim\pgf@xa>\pgf@xb
	\pgftransformyscale{-1}%
	\pgf@xc\pgf@xb
	\pgf@xb\pgf@xa
	\pgf@xa\pgf@xc
	\fi
	\pgfpathmoveto{\pgfqpoint{\qrr@tikz@sharp@z@}{.5\pgflinewidth}}%
	\pgfpathlineto{\pgfqpoint{\pgf@xa}{.5\pgflinewidth}}%
	\pgfpathlineto{\pgfqpoint{\pgf@ya}{+0pt}}%
	\pgfpathlineto{\pgfqpoint{\pgf@xb}{-.5\pgflinewidth}}%
	\pgfpathlineto{\pgfqpoint{\qrr@tikz@sharp@z@}{-.5\pgflinewidth}}%
	\pgfusepathqfill
}
\pgfarrowsdeclare{sharp <}{sharp <}{%
	\edef\pgf@marshal{\noexpand\pgfutil@in@{and}{\pgfgetarrowoptions{sharp <}}}%
	\pgf@marshal
	\ifpgfutil@in@
	\edef\pgf@tempa{\pgfgetarrowoptions{sharp <}}
	\expandafter\qrr@tikz@sharp@parse\pgf@tempa\@qrr@tikz@sharp@parse
	\else
	\expandafter\qrr@tikz@sharp@parse\pgfgetarrowoptions{sharp <}and-\pgfgetarrowoptions{sharp <}\@qrr@tikz@sharp@parse
	\fi
	\pgfmathparse{max(\pgf@tempa,\pgf@tempb,0)}%
	\let\qrr@tikz@sharp@max\pgfmathresult
	\pgfmathsetlength\pgf@xa{.5*\pgflinewidth * tan(\qrr@tikz@sharp@max)}%
	\pgfarrowsleftextend{+\pgf@xa}%
	\pgfarrowsrightextend{+\pgf@xa}%
}{%
	\edef\pgf@marshal{\noexpand\pgfutil@in@{and}{\pgfgetarrowoptions{sharp <}}}%
	\pgf@marshal
	\ifpgfutil@in@
	\edef\pgf@tempa{\pgfgetarrowoptions{sharp <}}
	\expandafter\qrr@tikz@sharp@parse\pgf@tempa\@qrr@tikz@sharp@parse
	\else
	\expandafter\qrr@tikz@sharp@parse\pgfgetarrowoptions{sharp <}and-\pgfgetarrowoptions{sharp <}\@qrr@tikz@sharp@parse
	\fi
	\pgfmathsetlength\pgf@ya{.5*\pgflinewidth * tan(max(\pgf@tempa,\pgf@tempb,0))}%
	\pgfmathsetlength\pgf@xa{-.5*\pgflinewidth * tan(\pgf@tempa)}%
	\pgfmathsetlength\pgf@xb{-.5*\pgflinewidth * tan(\pgf@tempb)}%
	\advance\pgf@xa\pgf@ya
	\advance\pgf@xb\pgf@ya
	\ifdim\pgf@xa>\pgf@xb
	\pgftransformyscale{-1}%
	\pgf@xc\pgf@xb
	\pgf@xb\pgf@xa
	\pgf@xa\pgf@xc
	\fi
	\pgfpathmoveto{\pgfqpoint{\qrr@tikz@sharp@z@}{.5\pgflinewidth}}%
	\pgfpathlineto{\pgfqpoint{\pgf@xa}{.5\pgflinewidth}}%
	\pgfpathlineto{\pgfqpoint{\pgf@ya}{+0pt}}%
	\pgfpathlineto{\pgfqpoint{\pgf@xb}{-.5\pgflinewidth}}%
	\pgfpathlineto{\pgfqpoint{\qrr@tikz@sharp@z@}{-.5\pgflinewidth}}%
	\pgfusepathqfill
}
\def\qrr@tikz@sharp@parse#1and#2\@qrr@tikz@sharp@parse{\def\pgf@tempa{#1}\def\pgf@tempb{#2}}

\title[On Macdonald expansions of $q$-chromatic symmetric functions]{On Macdonald expansions of $q$-chromatic symmetric functions and the Stanley--Stembridge Conjecture}
\author{Sean T.~Griffin}
\address{Fakultät für Mathematik, Universität Wien, Austria}
\email{sean.griffin@univie.ac.at}

\author{Anton Mellit}
\address{Fakultät für Mathematik, Universität Wien, Austria}
\email{anton.mellit@univie.ac.at}

\author{Marino Romero}
\address{Fakultät für Mathematik, Universität Wien, Austria}
\email{marino.romero@univie.ac.at}

\author{Kevin Weigl}
\address{Fakultät für Mathematik, Universität Wien, Austria}
\email{kevin.weigl@univie.ac.at}

\author{Joshua Jeishing Wen}
\address{Fakultät für Mathematik, Universität Wien, Austria}
\email{joshua.jeishing.wen@univie.ac.at}

\date{\today}

\begin{document}

\begin{abstract}
    The Stanley--Stembridge conjecture asserts that the chromatic symmetric function of a $(3+1)$-free graph is $e$-positive. Recently, Hikita proved this conjecture by giving an explicit $e$-expansion of the Shareshian--Wachs $q$-chromatic refinement for unit interval graphs. Using the $\mathbb{A}_{q,t}$ algebra, we give an expansion of these $q$-chromatic symmetric functions into Macdonald polynomials. Upon setting $t=1$, we obtain another proof of the Stanley--Stembridge conjecture and rederive Hikita's formula. Upon setting $t=0$, we obtain an expansion into Hall--Littlewood symmetric functions.
\end{abstract}

\maketitle

\section{Introduction}

In 1995, Stanley~\cite{Stanley-Chromatic} introduced a family of symmetric functions, called the \emph{chromatic symmetric functions}, attached to each finite simple graph which generalize the chromatic polynomial of the graph. The Stanley--Stembridge conjecture~\cite{Stanley-Stembridge}, a long-standing open problem in Algebraic Combinatorics, states that the chromatic symmetric function of a $(3+1)$-free graph is \emph{$e$-positive}, meaning that it expands positively in the elementary symmetric function basis. By \cite{Guay-Paquet}, the conjecture can be reduced to proving the $e$-positivity for unit interval graphs. Shareshian and Wachs~\cite{Shareshian-Wachs} subsequently generalized the conjecture by introducing a certain refinement, the \emph{chromatic quasisymmetric functions} (or \emph{$q$-chromatic symmetric functions}) $\chi_\e[X;q]$, which they conjecture are also $e$-positive.

Recently, Hikita~\cite{Hikita} proved the Stanley--Stembridge conjecture by giving a formula for the coefficients $c_{\e,\lambda}(q)$ in the $e$ expansion of $\chi_\e[X;q]$ as a sum of rational functions in $q$, which then specialize to nonnegative numbers upon taking $q=1$.

On the other hand, Carlsson and Mellit~\cite{Carlsson-Mellit}, in their proof of the long-standing Shuffle Conjecture, proved a plethystic formula relating $\chi_\e[X;q]$ to LLT polynomials, and they related the corresponding LLT polynomials to the action of a certain algebra, denoted by $\bA_{q,t}$, on symmetric functions. In \cite{Carlsson-Gorsky-Mellit} the action of $\bA_{q,t}$ was explicitly computed in a certain basis, generalizing the Macdonald basis.

In this article, we show that combining these results produces an explicit formula for the expansion of the LLT polynomials in the basis of modified Macdonald polynomials $\Ht_\lambda[X;q,t]$. Applying the plethysm leads to an explicit expansion of $\chi_\e(q)$ in the basis of integral form Macdonald polynomials. Setting $t=1$ turns the Macdonald polynomials into products of elementary symmetric functions, so we obtain an $e$-expansion of $\chi_\e[X;q]$. Thus we reprove Hikita's formula.
\\

For the reader's convenience, we mention the following conversion formulas between the modified Macdonald polynomials from \cite{GHT} and the integral form Macdonald polynomials from \cite{Macdonald-book}:
\begin{equation} \label{eq:HtoJ}
\Ht_\mu[(q-1) X; q, t] = q^{n(\mu')+|\mu|} J_{\mu'}[X;t,q^{-1}] = (-1)^{|\mu|}t^{n(\mu)} J_{\mu'}[X;t^{-1},q]
\end{equation}
where $n(\mu) = \sum_i (i-1)\mu_i$ and $|\mu| = \sum_i\mu_i$.

The main result can be stated as follows:
\begin{thm}\label{thm:IntroMainThm}
We have explicit tableau formulas for the coefficients $C_{\e,\mu}(q,t)$ appearing in the expansion
\begin{equation}
\chi_\e[X;q] = \sum_{\mu} C_{\e,\mu}(q,t) (q-1)^{-n} \Ht_\mu[(q-1)X; q, t],
\label{eq:MainThm}
\end{equation}
expanding $\chi_\e(q)$ in terms of plethystically evaluated modified Macdonald polynomials. 
In particular, when $t=1$, we get
\[
\chi_\e[X;q] = \sum_{\mu} C_{\e,\mu}(q,1) \left(\prod_{i}[\mu_i]_q!\right) e_\mu[X].
\]
Upon setting $q=1$, we find that the coefficient of $e_\mu$ is nonnegative, so the Stanley--Stembridge conjecture holds. Furthermore, the coefficient of $e_\mu$ is the same as the one found in \cite{Hikita}. 
\end{thm}
The coefficient of $e_\mu$ in $\chi_\e[X;q]$ can be counted by certain admissible proper colorings, see Corollary~\ref{cor:main_corollary} and Remark~\ref{rmk:colorings_e-expansion}.

It would be interesting to see if there is any connection between this formula and the $(q,t)$-chromatic symmetric functions very recently defined by Hikita~\cite{Hikita-qt}.

Note that $\Ht_{\mu'}[X;q,t]$ can be expanded into LLT polynomials~\cite{HHL}, and the integral form $J_{\mu}[X;q,t]$ can then be expanded into weighted sums of $q$-chromatic symmetric functions, see \cite{Haglund-Wilson} and \cite[Example 3.10]{Carlsson-Mellit}. In view of \eqref{eq:HtoJ}, our formula gives $q$-chromatic symmetric functions in terms of $J_\mu$. Combining the two, first expanding $J_\mu$ into $q$-chromatics and then back into $J_\mu$, one can obtain a curious formula for the identity matrix. 

\section{Background}

\subsection{Symmetric functions}
Let $\Lambda[X]$ be the space of symmetric functions with coefficients in $\mathbb{C}(q,t)$. Denote by $e_n$ and $h_n$ the elementary and complete homogeneous symmetric functions:
\begin{align*}
e_n=\sum_{i_1<\cdots<i_n}x_{i_1}\cdots x_{i_n},&& \text{ and } &&
h_n=\sum_{i_1\le \cdots\le i_n}x_{i_1}\cdots x_{i_n}.
\end{align*}
Given any integer partition $\lambda = (\lambda_1\geq \lambda_2\geq \cdots \geq \lambda_\ell>0)$, let $e_\lambda \coloneqq \prod_i e_{\lambda_i}$ and $h_\lambda \coloneqq \prod_i h_{\lambda_i}$. By standard symmetric function theory, $\{e_\lambda: \lambda\}$ and $\{h_\lambda: \lambda\}$ are bases of $\Lambda[X]$. We also denote by $\{s_\lambda:\lambda\}$ the \emph{Schur function} basis of $\Lambda[X]$. We denote by $\omega$ the involution on $\Lambda[X]$ which sends $h_\lambda$ to $e_\lambda$ or $s_\lambda$ to $s_{\lambda'}$, where $\lambda'$ is the conjugate of $\lambda$.  See~\cite{Macdonald-book} for more background on symmetric function theory.

In what follows, we will also make light use of ``plethystic notation'' and the theory of modified Macdonald polynomials. One can see \cite{Carlsson-Mellit} or \cite{GHT}, for example, for a more detailed explanation of these operations.

\subsection{The $\mathbb{A}_{q,t}$ algebra}

Our main method is to use the $\bA_{q,t}$ algebra, which was introduced by Carlsson and Mellit \cite{Carlsson-Mellit} to prove the Shuffle Theorem \cite{HHLRU} and its compositional refinement \cite{Haglund-Morse-Zabrocki}.

Let $V_k = \Lambda[X] \otimes \mathbb{C}[y_1,\dots,y_k]$. 
We will only need the following operators appearing in the $\mathbb{A}_{q,t}$ algebra.
\begin{definition}
For any expression $F(y_i, y_{i+1})$, let
\[
T_i F  = \frac{(q-1)y_i F(y_i,y_{i+1}) + (y_{i+1} - q y_i) F(y_{i+1},y_i)}{y_{i+1}-y_i}.
\]
Define $d_+:V_{k} \rightarrow V_{k+1}$ by setting
\[
d_+ F[X] = T_1 \cdots T_k F[X+(q-1) y_{k+1}],
\]
and define $d_- :V_{k+1} \rightarrow V_{k}$ by setting
\[
d_-F[X] = - F[X-(q-1)y_k] \Exp\left[ - y_k^{-1} X\right] \Big|_{y_k^{-1}},
\]
where $\Exp[X] = \sum_{n \geq 0} h_n[X]$ is the plethystic exponential, and $f\big|_{y_k^{-1}}$ means take the coefficient of $y_k^{-1}$ in $f$.
\end{definition}

\subsection{$q$-chromatic symmetric functions}

Given a finite graph $G$ with vertices $V = \{v_1,\dots, v_m\}$ and edges $E$, a \textbf{proper coloring of $G$} is a function $\kappa:V\to \mathbb{Z}_{>0}$ such that $\kappa(v)\neq \kappa(w)$ whenever $vw\in E$. Given a proper coloring $\kappa$ (and the ordering on $V$), the \textbf{ascents} of $\kappa$ are
\[
\asc(\kappa) \coloneqq \# \{(i<j) : v_i v_j\in E,\, \kappa(v_i) < \kappa(v_j)\}.
\]
The \textbf{$q$-chromatic symmetric function} introduced by Shareshian and Wachs~\cite{Shareshian-Wachs} is given by
\[
\chi_G[X;q] \coloneqq \sum_{\kappa}q^{\asc(\kappa)} x^\kappa,
\]
where the sum is over all proper colorings of $G$ and $x^\kappa = \prod_i x_{\kappa(v_i)}$. 

A \textbf{unit interval graph} is a particular type of graph whose $q$-chromatic symmetric function has particularly nice properties. One equivalent way of defining such a graph is as follows: Let $[n]=\{1,\ldots,n\}$ and $\e: [n]\to \{0,1,\dots, n-1\}$ be a \textbf{reverse hessenberg function}, which is a function that is weakly increasing such that $\e(i) < i$ for all $i$. Then the unit interval graph $\Gamma=\Gamma_\e$ is the graph with vertex set $V = [n]$ and edges $\{ij: \e(j) < i < j\}$. We will write $\chi_\e[X;q]=\chi_{\Gamma_\e}[X;q]$.

Before Hikita's proof of $e$-positivity, it was known that $\chi_\e[X;q]$ is a Schur positive symmetric function by Shareshian and Wachs \cite{Shareshian-Wachs} who refined a formula for $\chi_\e[X;1]$ of Gasharov~\cite{Gasharov}. Brosnan and Chow~\cite{Brosnan-Chow} subsequently proved that $\chi_\e[X;q]$ is the Frobenius characteristic of the dot action of $S_n$ on the singular cohomology ring of a regular semisimple Hessenberg variety (up to tensoring by the sign representation), a conjecture that was posed by Shareshian and Wachs, giving another proof of Schur positivity.

A Dyck path $D= W^{a_1} S^{b_1} \cdots W^{a_\ell} S^{b_\ell}$  is a sequence of West and South unit lattice steps satisfying
\begin{enumerate}
    \item $a_1 - b_1 +\cdots + a_r - b_r \geq 0$ for all $r$, and
    \item $a_1+\cdots + a_\ell  = b_1 +\cdots + b_\ell = n$, for which we then say that $D \in D_n$ is a Dyck path in the $n \times n$ square.
\end{enumerate}
We visualize a Dyck path as a lattice path starting at $(n,n)$ and taking West $(-1,0)$ and South $(0,-1)$ steps according to the word $D$, which stays weakly above the diagonal of the $n\times n$ grid by (1) and ends at $(0,0)$ by (2). 

Let $\e_D(i)$ be the number of South steps preceding the $i$-th West step of $D$. In this way, we associate to every $D\in D_n$ a reverse Hessenberg function $\e_D$ of length $n$. Any reverse Hessenberg function arises from a unique Dyck path in this way. So we will identify Dyck paths with reverse Hessenberg functions and write $\e\in D_n$. See Figure~\ref{fig:Dyckpath} for $D$ and the unit interval graph corresponding to $\e = (0,1,1,2,2,3,4)$. 
\\

\begin{figure}[t]
\centering
\includegraphics{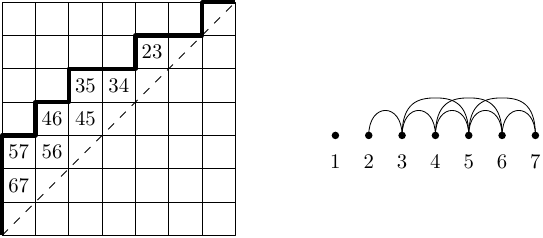}
\caption{The Dyck path $D$ and unit interval graph $\Gamma_\e$ associated to the reverse hessenberg function $\e = (0,1,1,2,2,3,4)$.
Cells are labeled by the edges, $ij$, present in the graph.}
\label{fig:Dyckpath}
\end{figure}

Carlsson and Mellit showed that the $q$-chromatic symmetric function can be obtained from $d_+$ and $d_-$ by first writing
\[
F_D[X;q]  \coloneqq d_-^{b_\ell} d_+^{a_\ell} \cdots d_-^{b_1} d_+^{a_1} (1),
\]
then setting
\[
\chi_D[X;q] \coloneqq (q-1)^{-n} F_D[(q-1)X].
\]
\begin{prop}[{\cite[Theorem 4.4$+$Proposition 3.5]{Carlsson-Mellit}}]
For any reverse Hessenberg function $\e$ with the corresponding Dyck path $D \in D_n$, we have
\[
\chi_D[X;q] = \chi_{\e}[X;q].
\]
\end{prop}

We will write $F_\e$ for the corresponding function $F_D$ where $D$ is such that $\e=\e_D$.
For a reference on LLT polynomials and $q$-chromatic symmetric functions, see \cite{Alex-Panova}.
\subsection{Macdonald symmetric functions}
The modified Macdonald polynomials $\Ht_\mu[X;q,t] = \Ht_\mu[X]$ specialize to modified homogeneous symmetric functions:
\[
\Ht_\mu[X;q,1] = (q;q)_\mu h_\mu \left[ \frac{X}{1-q} \right],
\]
where $(q;q)_\mu = (q;q)_{\mu_1} \cdots (q;q)_{\mu_{\ell(\mu)}} = h_\mu[1/(1-q)]^{-1}$, and $(q;t)_r= (1-q)(1-q t)\cdots (1-q t^{r-1})$ is the Pochhammer symbol. 
Therefore, 
\[
\Ht_\mu[(q-1)X;q,1] = (-1)^{|\mu|} (q;q)_\mu e_\mu[X],
\]
and as a consequence, the elementary basis expansion of $\chi_\e[X;q]$ can be found by the expansion of $F_\e$ in the basis $\{h_\mu[X/(1-q)]\}_\mu$. 

Our goal will be to first write $F_\e$ in terms of the modified Macdonald basis:
\[
F_\e[X;q] = \sum_{\mu} C_{\e,\mu}(q,t) ~\Ht_\mu[X;q,t].
\]
Since $\chi_\e[X;q]$ has no $t$ variable, we must then also have
\begin{align*}
F_\e[X;q] &= \sum_{\mu} C_{\e,\mu}(q,1) ~\Ht_\mu[X;q,1]
\\
&= \sum_{\mu} C_{\e,\mu}(q,1) (q;q)_\mu h_\mu\left[ \frac{X}{1-q} \right].
\end{align*}
From here, it follows that 
\begin{equation}\label{eq:ChiEExpansion}
\chi_\e[X;q] = \sum_{\mu} C_{\e,\mu}(q,1) (1-q)^{-|\mu|}(q;q)_\mu e_\mu[X],
\end{equation}
where we have used that $h_\mu[-X] = (-1)^{|\mu|}\omega (h_\mu)[X] = (-1)^{|\mu|}e_\mu$.

The Stanley--Stembridge Conjecture then follows from combining~\eqref{eq:ChiEExpansion} with the next proposition.
\begin{prop}\label{prop:MainProp}
For any reverse Hessenberg function $\e$ of length $n$ and partition $\mu$ of $n$, $C_{\e,\mu}(q,1)$ is positive when $q=1$. 
Furthermore, $C_{\e,\mu}(q,1) (1-q)^{-|\mu|}(q;q)_\mu$ equals the formula for the coefficient of $e_\mu$ given in \cite{Hikita}.
\end{prop}

\subsection{The action of $d_{\pm}$ on $I_{\la,w}$}
We now think of partitions as being given by a sum of cells. We say $(i,j) \in \lambda$ if $0\leq j <\ell(\lambda)$ and $0 \leq i < \lambda_{j+1}$. To simplify notation, we will then write
\[
\lambda = \sum_{(i,j) \in \lambda} q^i t^j.
\]
For a given $\lambda$, we say $x = q^r t^s$ is an \textbf{addable cell} if $\lambda+x = \mu$ is a partition, and we will then write $x \in \Add(\lambda)$. The cells in $\lambda$ that are in the same row as $x$ are denoted by $\mathcal{R}_{\lambda,x}$; and similarly, the cells in $\lambda$ that are in the same column as $x$ are denoted by $\mathcal{C}_{\lambda,x}$. For a given cell $c \in \lambda$, let $a_\lambda(c)$ denote the arm length of $c$ in $\lambda$ (that is, the number of cells to its right), and let $l_\lambda(c)$ be its leg length (that is, the number of cells above $c$ in the French convention).

If $x \in \Add(\lambda)$, we let
\[
d_{\lambda,x}= \prod_{c\in \mathcal{R}_{\lambda,x}} \frac{ q^{a_\lambda(c)} - t^{l_{\lambda}(c)+1}}{q^{a_\lambda(c)+1} - t^{l_\lambda(c)+1}}
\prod_{c\in \mathcal{C}_{\lambda,x}} \frac{ q^{a_\lambda(c)+1} - t^{l_{\lambda}(c)}}{q^{a_\lambda(c)+1} - t^{l_\lambda(c)+1}}.
\]
These terms appear in the Pieri coefficients for Macdonald polynomials \cite{GHXZ}:
\[
e_1 \Ht_\lambda[X] = \sum_{x\in\Add(\lambda)} d_{\lambda,x} \Ht_{\lambda+x}[X].
\]

A basis for $V_k$ is given by certain elements of the form $I_{\lambda,w}$, where $w = (w_1,\dots, w_k)$ is a sequence of cells in $\lambda$ which form a horizontal strip, meaning no two cells are in the same column and each row is connected and right justified in $\lambda$. The list is written such that $\lambda - w_1 - w_{2} - \cdots - w_{s}$ forms a partition for each $s\leq k$, meaning that if we place $i$ in the cell $w_{k-i +1}$, we get a standard tableau of the horizontal strip $w_1+\cdots + w_k$. 

In~\cite{Carlsson-Gorsky-Mellit}, Carlsson, Gorsky, and Mellit construct an action of $\bA_{q,t}$ on the direct sum of the (localized) equivariant $K$-theory of parabolic flag Hilbert schemes, $K_{\bC^*\times \bC^*}(\mathrm{PFH}_{n,n-k})$; moreover, they show that
\[
\bigoplus_{k}K_{\bC^*\times \bC^*}(\mathrm{PFH}_{n,n-k})\cong\bigoplus_k V_k
\]
as $\bA_{q,t}$-modules.
The basis $\{I_{\lambda,w}\}$ corresponds to the distinguished basis of $K_{\bC^*\times \bC^*}(\mathrm{PFH}_{n,n-k})$ given by the classes of torus fixed points. It will not be important to know explicitly what the elements $I_{\lambda,w}$ are. Instead, we rely on the following three facts. 
\begin{prop}[\cite{Carlsson-Gorsky-Mellit}] \label{prop:Aqt-action}
Given a partition $\lambda$ and sequence $w=(w_1,\dots,w_k)$ forming a horizontal strip (as described above), we have
\begin{gather}
I_{\lambda,\emptyset} = \Ht_\lambda[X]/\Ht_\lambda[-1] \in V_0,\\
d_- I_{\lambda,wx}= I_{\lambda,w}, \text{ and } \\
d_+ I_{\lambda,w} = -q^k \sum_{x\in \Add(\lambda)} x d_{\lambda,x} \left(\prod_{i=1}^k \frac{x-tw_i}{x-qtw_i} \right) I_{\lambda +x,xw}.
\end{gather}
\end{prop}

We will also need the following identity:
\begin{equation}\label{eq:HAtMinus1}
\Ht_\mu[-1] = (-1)^{|\mu|}\prod_{(i,j)\in \mu} q^i t^j.
\end{equation}

\section{Main proof}

In this section, we prove Theorem~\ref{thm:IntroMainThm} using the $\bA_{q,t}$ algebra.

\begin{notation}
Let $\e$ be a reverse Hessenberg function.
 If $ \e(j) < i < j,$ then we will write $i \prec j$. Note that  $\prec$ is not a partial order since it is not necessarily transitive.
\end{notation}
For example, in the case $\e=(0,1,1,2,2,3,4)$ considered in Figure \ref{fig:Dyckpath}, the pairs $i\prec j$ are given by the edges of $\Gamma_\e$. Namely, $2\prec 3$, $3\prec 4$, $3\prec 5$ and so on. This function corresponds to $\e=\e_D$ for the Dyck path $D=WSWWSWWSWSWSSS$, and the pairs $i\prec j$ are in bijection with the cells in the $7\times 7$ grid between the diagonal and $D$. 

Let $\SYT_\mu$ be the set of standard Young tableaux of shape $\mu$. 
For a given $j$, let $T_{< j}$ be the standard tableau given by the entries smaller than $j$. We will also let $T_{\prec j}$ be the skew-shaped tableau with entries $i$ satisfying $i \prec j$. 
Similarly, define $T_{\leq j}$ and $T_{\preceq j}$ to include the entry $j$. 
We will say that a skew shape is {\emph{strict}} if its entries are increasing along rows and columns. 
A strict skew tableau $T$ gives a sequence $w(T) = (w_1,\dots,w_k)$ of cells, where $w_1$ is the cell with the largest entry, $w_2$ is the cell with the second largest entry, and so on until we reach the cell $w_k$ with the smallest entry.
For a strict skew tableau, we will also write $\shape(T)$ to denote the skew diagram with no fillings.
\\

First, we write 
\[F_\e[X;q]= d_-^{r_n} d_+ d_-^{r_{n-1}} \cdots d_+ d_-^{r_1} d_+ (I_{\emptyset,\emptyset}),
\]
where $r_1+\cdots +r_j = \e(j+1)$, $r_1+\cdots+r_n=n$. Note that $\e(1)=0$. 
Suppose that after we apply $d_+ d_-^{r_{j-1}} \cdots d_-^{r_1} d_+ $ we arrive at a term with $I_{\lambda, w}$. Applying $d_{-}^{r_j}$ eliminates $r_j$ elements from $w$ giving $w'$, leaving only the last $j - \e(j+1)$ cells which were added to $\lambda$; and $j - \e(j+1)$ is the number of $i \prec j+1$. Applying $d_+$ then adds a cell $x$ to $\lambda$ so that the sequence $xw$ forms a strict horizontal strip. 
We interpret this combinatorially in the following way.

Let
\begin{equation}
 A^{\lambda,w}_x \coloneqq 
  \prod_{c\in \mathcal{R}_{\lambda,x}} \frac{ q^{a_\lambda(c)} - t^{l_{\lambda}(c)+1}}{q^{a_\lambda(c)+1} - t^{l_\lambda(c)+1}}
\,\times\,\prod_{c\in \mathcal{C}_{\lambda,x}} \frac{ q^{a_\lambda(c)+1} - t^{l_{\lambda}(c)}}{q^{a_\lambda(c)+1} - t^{l_\lambda(c)+1}}
\,\times\, \prod_{i=1}^k \frac{x-tw_i}{x-qtw_i}.\label{eq:Axlamw}
\end{equation}
Let $\SYT^\e_\lambda$, be the set of standard tableaux $T$ such that for each $i$, $T_{\preceq i}$ is a horizontal strip. 
\begin{rmk}\label{rmk:ll increase}
We may alternatively define $\SYT^\e_\la$ as follows. Define a partial order on $[n]$, called the \emph{unit interval order}, by $i\ll j$ if $i \leq \e(j)$. Then $\SYT^\e_\la$ is the set of standard tableaux whose rows increase with respect to $<$ and whose columns increase with respect to $\ll$.
\end{rmk}
Let $T^{(i)} \coloneqq \shape(T_{\leq i})- \shape(T_{<i})$ be the cell containing $i$, and let
\[
A_{T}^{\e}(q,t) \coloneqq \prod_{i=1}^n A^{\shape(T_{< i}),w(T_{\prec i})}_{T^{(i)}}.
\]
Lastly, note that
\[
\prod_{i=1}^n -q^{|T_{\prec i}|} T^{(i)}
= q^{\binom{n}{2}-|\e|} \Ht_{\lambda}[-1],
\]
where $|\e| = \sum_i \e(i)$. Proposition \ref{prop:Aqt-action} then tells us that
\[
F_\e[X;q] = \sum_\lambda \Ht_\lambda[X]/\Ht_{\lambda}[-1] ~ \left(  \sum_{T \in \SYT^\e_\lambda}  \prod_{i=1}^n -q^{|T_{\prec i}|} T^{(i)} A^{\shape(T_{< i}),w(T_{\prec i})}_{T^{(i)}}  \right).
\]
This gives us the next proposition.
\begin{prop}\label{prop:FDExpansion}
We have the following expansions of $F_\e$ and $\chi_\e$:
\begin{align}\label{eq:MasterFormulaF}
F_\e[X;q] &= q^{\binom{n}{2}-|\e|}~ \sum_\lambda ~
 \left(\sum_{T \in \SYT^\e_\lambda} A_T^{\e}(q,t)\right) \Ht_\lambda[X;q,t],\\
\chi_\e[X;q] &= (q-1)^{-n} q^{\binom{n}{2}-|\e|}~ \sum_\la ~  \left(\sum_{T \in \SYT^\e_\lambda} A_T^{\e}(q,t)\right) \Ht_\la[(q-1)X;q,t].\label{eq:MasterFormulaChi}
\end{align}
\end{prop}
Even though we will not use the following version of these expansions, we note here that there is an alternate expression for Proposition \ref{prop:FDExpansion}
which we expect to have a geometric interpretation.
\begin{prop}
We have
\[
F_\e[X;q] = \sum_{\mu}  C_{\e,\mu}(q,t) ~\Ht_\mu[X],
\]
where 
\[
C_{\e,\mu}(q,t) =\frac{q^{\binom{n}{2} - |\e|}}{w_\mu} \sum_{T \in \SYT^\e_\mu}
\prod_i (1-qtT^{(i)}) \prod_{i<j} \frac{1-qT^{(j)}/T^{(i)}}{1-T^{(j)}/T^{(i)}} 
\times \prod_{i \not\gg j} 
\frac{1-tT^{(j)}/T^{(i)}}{1-qtT^{(j)}/T^{(i)}},
\]
and 
\[
w_\mu = \Ht_\mu[-1] \prod_{i} (1-1/T^{(i)})(1-qtT^{(i)}) \prod_{i,j} \frac{(1-qT^{(j)}/T^{(i)})(1-tT^{(j)}/T^{(i)})}
{
(1-T^{(j)}/T^{(i)})(1-qtT^{(j)}/T^{(i)})
}.
\]
The products here are understood by removing all terms in which we have $1-1=0$, since these terms all cancel.
\end{prop}

Let
$\overline{\SYT}^\e_\mu$ denote all standard tableaux in $\SYT^\e_\mu$ with the following extra condition:
\begin{equation}\label{condition}
\text{For every $i$ not in the first column, there is a $j$ in the column immediately left of $i$ such that $j \prec i$.}\tag{$\ast$}
\end{equation}

\begin{rmk}
\label{rmk:star_and_ll_conditions}
Having condition \eqref{condition} in place, together with the property that the columns increase with respect to $\ll$ (see Remark \ref{rmk:ll increase}), the standard tableau conditions hold automatically. The fact that columns are increasing is clear. To see that the row conditions hold, consider two horizontally neighboring entries $a, i$. By contradiction, suppose $a>i$ and that this pair is the lowest one satisfying this condition (in the French convention). By \eqref{condition}, there exists $j$ in the same column as $a$ satisfying $j\prec i$. In particular, we have $j<i$. If $a\leq j$, then we have $a\leq j<i$, a contradiction. So we have $a>j$. Since $a$ and $j$ are in the same column, we must have $a\gg j$, and $j$ is below $a$. Let $j'$ be the entry directly to the right of $j$. Since $a,i$ have been chosen to be the lowest, we have $j<j'$. This implies $j<j'\ll i$, which contradicts the condition $j\prec i$.
\end{rmk}

\begin{example}\label{ex:tableau331}
When $\e =(0, 1, 1, 2, 2, 3, 4)$,
\begin{align*}
\scalebox{.7}{
\begin{tikzpicture}[scale=1]
\draw (-1,1.5) node {\LARGE $T=$};
\partitionfr{3,3,1};
\draw (.5,.5) node {\LARGE$1$};
\draw (.5,1.5) node {\LARGE$2$};
\draw (1.5,.5) node {\LARGE$3$};
\draw (2.5,.5) node {\LARGE$4$};
\draw (.5,2.5) node {\LARGE$5$};
\draw (1.5,1.5) node {\LARGE$6$};
\draw (2.5,1.5) node {\LARGE$7$};
\end{tikzpicture}}
\end{align*}
is an element of $\SYT^\e_{3,3,1}$ since $T_{\preceq i}$ is a horizontal strip for each $i$. 
In order to satisfy \eqref{condition}, $2$ is forced to be in the first column since $1\not\prec 2$. Furthermore, since $2\prec 3$, $3\prec 4$, $5\prec 6$, and $6\prec 7$, all entries $i$ satisfy condition $\eqref{condition}$. Hence, $T$ is in $\overline{\SYT}^\e_{(3,3,1)}$.
\end{example}

\begin{lem}
For $T \in \SYT^\e_\mu$, $A_T^{\e}(q,1) \neq 0$ if and only if $T \in \overline{\SYT}^\e_\mu$.
\end{lem}
\begin{proof} Given $\lambda$, $w$, and $x = q^r t^s$ as before, 
we have
\begin{align*}
 A^{\lambda,w}_x(q,1)
&=  \frac{\prod_{c\in \mathcal{R}_{\lambda,x}: a_\lambda(c) >0} q^{a_\lambda(c)} - 1}{\prod_{c\in \mathcal{R}_{\lambda,x}} q^{a_\lambda(c)+1} - 1}\quad\times\prod_{c\in \mathcal{C}_{\lambda,x}} \frac{ q^{a_\lambda(c)+1} - 1}{q^{a_\lambda(c)+1} - 1}\\
& \hspace{3em}\quad\times \left.\left(\prod_{c\in \mathcal{R}_{\lambda,x}: a_\lambda(c) =0} (1-t^{l_\lambda(c)+1})
\times\prod_{i=1}^k \frac{x-tw_i}{x-qtw_i} \right) \right|_{t=1}
.
\end{align*}
When $r=0$, we have
\begin{align*}
 A^{\lambda,w}_x(q,1)
&=   \left.\left( \prod_{c\in \mathcal{R}_{\lambda,x}: a_\lambda(c) =0} (1-t^{l_\lambda(c)+1})
\times \prod_{i=1}^k \frac{q^r t^s -tw_i}{q^r t^s -qtw_i} \right) \right|_{t=1}.
\end{align*}
If $r>0$, we get
\begin{align*}
 A^{\lambda,w}_x(q,1)
&=  \frac{1}{q^r -1} \quad\times \left.\left(\prod_{c\in \mathcal{R}_{\lambda,x}: a_\lambda(c) =0} (1-t^{l_\lambda(c)+1})
\times \prod_{i=1}^k \frac{q^r t^s -tw_i}{q^r t^s -qtw_i} \right) \right|_{t=1}.
\end{align*}

The first product, over cells in $\mathcal{R}_{\lambda,x}$ with arm length $0$, has at most one element. 
Let $c_i$ be the column for $w_i$, so that $w_i|_{t=1} = q^{c_i}$. 
If $r = 0$, then $x$ is in the first column, so the first product is empty and we can evaluate the second product to get
\begin{equation}\label{eq:RZero}
 A^{\lambda,w}_x(q,1)
=   \prod_{i=1}^k \frac{1 -q^{c_i}}{1 -q^{c_i+1}} ,
\end{equation}
which is a nonzero function of $q$ since $c_i>0$ for all $i$ by the fact that $xw$ is a horizontal strip.

If $r > 0$, then
\begin{align*}
 A^{\lambda,w}_x(q,1)
= \frac{1}{q^r -1}\times\left. \left(1-t^{l_\lambda(q^{r-1}t^s)+1} \right)
\times \prod_{i=1}^k \frac{q^r t^s -tw_i}{q^r t^s -qtw_i}  \right|_{t=1}.
\end{align*}
If there is no $c_i = r-1$, then the denominator does not contribute a pole at $t=1$, giving $A^{\lambda,w}_x (q,1) = 0$. 
\\

On the other hand, if $w_j = q^{r-1}t^{s'}$, then $l_\lambda(q^{r-1}t^s) = s'-s$ and
\begin{align}
 A^{\lambda,w}_x(q,1)
&=   \frac{1}{q^r -1} \prod_{i=1:i \neq j}^k \frac{q^r  - q^{c_i}}{q^r  -q^{c_i +1}} \quad \times   \left.\left( 1-t^{l_\lambda(q^{r-1}t^s)+1} \right)
\frac{q^r t^s -tw_j}{q^r t^s -qtw_j}  \right|_{t=1} \\
&=
    \frac{1}{q^r -1} \prod_{i=1:i \neq j}^k \frac{q^r  - q^{c_i}}{q^r  -q^{c_i +1}} \quad \times  q^{-r}t^{-s} \left.\left( 1-t^{s'-s+1} \right)
\frac{q^r t^s -tq^{r-1}t^{s'}}{1  - t^{s'-s+1}}  \right|_{t=1} \\
& 
=
q^{- 1}   \frac{q-1}{q^r -1} \prod_{i=1:i \neq j}^k \frac{q^r  - q^{c_i}}{q^r  -q^{c_i +1}},\label{eq:RPositive}
\end{align}
which is a nonzero function of $q$ since $r\neq c_i$ by the fact that $xw$ is a horizontal strip. This proves the proposition.
\end{proof}
The computation above gives an explicit expansion for $A^{\lambda,w}_x(q,1)$, and therefore, an explicit expansion of $\chi(\e;q)$.
\begin{cor} \label{cor:main_corollary}
For a given standard tableau $T$, let $ c_i(T)$ be the column of $i$, so that $T^{(i)}|_{t=1} = q^{c_i(T)}$.
Then
\begin{equation}\label{eq:FirstPositiveExpansion}
\chi_\e[X;q]  = q^{\binom{n}{2}- n-|\e| } 
\sum_{\lambda}
\left( q^{\ell(\lambda)}
[\lambda_1]_q \cdots [\lambda_{\ell(\lambda)}]_q
\sum_{T \in \overline{\SYT}^\e_\lambda}
\prod_{ \substack{ i \prec j \\ c_j(T) \neq c_{i}(T)+1}}
\frac{q^{c_j(T)} - q^{c_i(T)}}{q^{c_j(T)} - q^{c_i(T)+1}}\right) e_\lambda[X].
\end{equation}
Moreover, setting $q=1$ in this identity yields a positive expansion of $\chi_\e[X;q]$ into the elementary basis. This proves the Stanley--Stembridge conjecture. Explicitly:
\begin{equation}\label{eq:FirstPositiveExpansion1}
\chi_\e[X;1]  = 
\sum_{\lambda}
\left(
\lambda_1 \cdots \lambda_{\ell(\lambda)}
\sum_{T \in \overline{\SYT}^\e_\lambda}
\prod_{ \substack{ i \prec j \\ c_j(T) \neq c_{i}(T)+1}}
\frac{|c_j(T) - c_i(T)|}{|c_j(T)-c_i(T)-1|}\right) e_\lambda[X].
\end{equation}
\end{cor}
\begin{rmk} \label{rmk:colorings_e-expansion}
One can also interpret this formula in terms of proper colorings. We call a proper coloring $\kappa$ on the graph $\Gamma_\e$ \emph{admissible} if for every $i$ with $\kappa(i)>1$, there is an edge $v_j v_i$ (with $j<i$) for which $\kappa(j) = \kappa(i)-1$.
\\
Let $K^\e_{\lambda'}$ be the set of all admissible colorings $\kappa$ on the graph $\Gamma_\e$ such that $\lambda_i'$ is the multiplicity of the color $i$. To such a coloring we associate a weight
\[
\weight(\kappa) = \prod_{ \substack{ \text{edges $v_i v_j$,$i<j$} \\ \kappa(j) \neq \kappa(i)+1}}
\frac{|\kappa(j) - \kappa(i)|}{|\kappa(j)-\kappa(i)-1|}.
\]
Then \eqref{eq:FirstPositiveExpansion1} translates to
\[
\chi_\e[X;1] = \sum_{ \lambda}~ e_\lambda~ \sum_{\kappa  \in K^\e_{\lambda'}} \lambda_1\cdots \lambda_{\ell(\lambda)} \weight(\kappa).
\]
The $q$-version for colorings is also directly obtained from \eqref{eq:FirstPositiveExpansion} by assigning to each admissible coloring a $q$-weight appropriately.

To see why this is equivalent, one interprets $T\in \overline{\SYT}^\e_\lambda$ as a coloring $\kappa_T$ where $i$ is assigned the color $\kappa_T(i) = c_i(T)+1$. If $i \prec j$, $i$ and $j$ cannot be in the same column. This means the coloring is proper. Condition~\eqref{condition} translates to the second condition of $K^\e_{\lambda'}.$ By Remark~\ref{rmk:star_and_ll_conditions}, the two descriptions are equivalent.
\end{rmk}
\begin{proof}[Proof of Corollary \ref{cor:main_corollary}]
Using Equations \eqref{eq:RZero} and \eqref{eq:RPositive}, we get that for $T \in \overline{\SYT}^\e_\lambda$,
\begin{align*}
A^\e_T(q,1) & = \prod_{ \substack{i \prec j \\ c_j(T) \neq c_i(T) +1}} \frac{q^{c_j(T)} - q^{c_i(T)}}{q^{c_j(T)} - q^{c_i(T) +1}} \times \prod_{c_j(T) > 0 } q^{-1} \frac{q-1}{q^{c_j(T)} -1}  \\
& = \frac{q^{\ell(\lambda)-n}}{ [\lambda_1-1]_q! \cdots[\lambda_{\ell(\lambda)}-1]_q !} \prod_{ \substack{i \prec j \\ c_j(T) \neq c_i(T) +1}} \frac{q^{c_j(T)} - q^{c_i(T)}}{q^{c_j(T)} - q^{c_i(T) +1}}.
\end{align*}
Combining this with
\eqref{eq:ChiEExpansion} and using $(1-q)^{-n} (q;q)_\la = \prod_i [\la_i]_q!$ gives the above expansion.

To see that \eqref{eq:FirstPositiveExpansion} specializes to a positive $e_\lambda$ expansion at $q=1$, simply observe that for all $i\prec j$ with $c_j(T)\neq c_i(T)+1$, either $c_j(T) < c_i(T)$ or $c_j(T) > c_i(T)+1$. In the first case,
\[
\frac{q^{c_j(T)} - q^{c_i(T)}}{q^{c_j(T)} - q^{c_i(T) +1}} = \frac{[c_i(T)-c_j(T)]_q}{[c_i(T)+1-c_j(T)]_q}.
\]
In the second case, 
\[
\frac{q^{c_j(T)} - q^{c_i(T)}}{q^{c_j(T)} - q^{c_i(T) +1}} = q^{-1}\frac{[c_j(T)-c_i(T)]_q}{[c_j(T)-c_i(T)-1]_q}.
\]
In either case, setting $q=1$ produces $\frac{|c_j(T) - c_i(T)|}{|c_j(T)-c_i(T)-1|}$ and we obtain \eqref{eq:FirstPositiveExpansion1}.
\end{proof}

\begin{example}
For $\e$ in Figure \ref{fig:Dyckpath}, the $q$-chromatic symmetric function looks as follows:
\[
\chi_{(0,1,1,2,2,3,4)} = \left(7 q^{8} + 70 q^{7} + 364 q^{6} + 1162 q^{5} + 1834 q^{4} + 1162 q^{3} + 364 q^{2} + 70 q + 7\right) m_{1,1,1,1,1,1,1} 
\]
\[
+ \left(q^{8} + 15 q^{7} + 97 q^{6} + 361 q^{5} + 612 q^{4} + 361 q^{3} + 97 q^{2} + 15 q + 1\right) m_{2,1,1,1,1,1} 
\]
\[
+ \left(2 q^{7} + 21 q^{6} + 102 q^{5} + 194 q^{4} + 102 q^{3} + 21 q^{2} + 2 q\right) m_{2,2,1,1,1} + \left(3 q^{6} + 25 q^{5} + 58 q^{4} + 25 q^{3} + 3 q^{2}\right) m_{2,2,2,1} 
\]
\[
+ \left(q^{7} + 9 q^{6} + 43 q^{5} + 86 q^{4} + 43 q^{3} + 9 q^{2} + q\right) m_{3,1,1,1,1} + \left(q^{6} + 9 q^{5} + 24 q^{4} + 9 q^{3} + q^{2}\right) m_{3,2,1,1} + 
\]
\[
\left(q^{5} + 6 q^{4} + q^{3}\right) m_{3,2,2} + 2 q^{4} m_{3,3,1} + \left(q^{5} + 4 q^{4} + q^{3}\right) m_{4,1,1,1} + q^{4} m_{4,2,1}.
\]
The $e$-expansion is
\[
\chi_{(0,1,1,2,2,3,4)} = q^{4} e_{3,2,1,1} + \left(q^{5} + q^{4} + q^{3}\right) e_{3,3,1} + \left(q^{5} + q^{4} + q^{3}\right) e_{4,1,1,1} 
\]
\[
+ \left(q^{6} + 4 q^{5} + 6 q^{4} + 4 q^{3} + q^{2}\right) e_{4,2,1} + \left(q^{7} + 5 q^{6} + 8 q^{5} + 9 q^{4} + 8 q^{3} + 5 q^{2} + q\right) e_{5,1,1} 
\]
\[
+ \left(q^{8} + 4 q^{7} + 7 q^{6} + 8 q^{5} + 8 q^{4} + 8 q^{3} + 7 q^{2} + 4 q + 1\right) e_{6,1}.
\]
The coefficient in front of $e_{3,3,1}$ comes from a single tableau, see Example \ref{ex:tableau331}. Equation \eqref{eq:FirstPositiveExpansion} produces
\[
q^{21-7-13+3} [3]_q [3]_q [1]_q \times \frac{1-q}{1-q^2} \times \frac{1-q^2}{1-q^3} \times \frac{q-q^2}{q-q^3}\times \frac{q^2-1}{q^2-q},
\]
the fractions coming from the pairs $(i,j)=(3,5), (4,5), (4,6), (5,7)$. Performing cancellations results in $q^3 [3]_q$. Similarly, the coefficients in front of $e_{3,2,1,1}$, $e_{4,1,1,1}$ and $e_{6,1}$ each come from a single tableau. The coefficient in front of $e_{4,2,1}$ comes from $3$ tableaux and we observe non-trivial denominators:
\[
q^2 [4]_q [2]_q + q^3 \frac{[4]_q [2]_q}{[3]_q} + q^3 \frac{[2]_q^4}{[3]_q}.
\]
For instance, setting $q=1$ produces  a decomposition of $16$ as $16=8 + \frac{8}{3} + \frac{16}{3}$. The coefficient in front of $e_{5,1,1}$ is decomposed as a sum over $4$ tableaux, two of the summands having $[3]_q$ in the denominator.
\end{example} 

\subsection{The relation to Hikita's formula}

We will now find an equivalent expression that lets us re-derive Hikita's formula. 
Let, once again, $\lambda,w,$ and $x$ be given as in Proposition \ref{prop:Aqt-action}.
Let
$c_1<\cdots< c_k$ be the columns for $w = (w_1,\dots, w_k)$ in increasing order. Let $w^{L}_1< \cdots < w^L_m$ be the subsequence of columns such that $w^L_j -1 \neq c_i$ for some $i$, and let $w^{R}_1< \cdots < w^R_m$ be the subsequence of columns for which $w^R_{j}+1 \neq c_i$ for some $i$.
In other words, the sequence of columns $c_1<\cdots < c_k$ consists of a union of contiguous segments, and the left-most elements of these contiguous segments are denoted by $w_i^{L}$ and the right-most elements of these contiguous segments are denoted by $w_i^R$.

\begin{lem}\label{lem:Aq1}
Let $x = q^r t^s$ as in Proposition~\ref{prop:FDExpansion}. When $r=0$,
\begin{equation}\label{eq:RZeroLem}
A^{\lambda,w}_x(q,1) =   \prod_{i=1}^m \frac{[w^{L}_i]_q}{[w^R_i+1]_q}  .
\end{equation}
When $r>0$, letting $j$ be maximal such that $w^R_j$ is on the left of $x$ (in other words, $w^R_j = r-1$), then
\begin{equation}\label{eq:RPositiveLem}
A^{\lambda, w}_x(q,1)= 
q^{b(w,x)-r} \prod_{i=1}^j \frac{[r-w^L_i]_q}{[r-w^{R}_{i-1} -1]_q}
\prod_{i=j+1}^m \frac{[w^L_i-r]_q}{[w^{R}_i +1-r]_q},
 \end{equation}
where 
$w_0^R\coloneqq -1$, and $b(w,x) = \sum_{i=1}^{j}(w^L_{i}-w^R_{i-1}-1)$ counts the number of columns left of $x$ with no $w_j$. 
\end{lem}

\begin{proof}
When $r=0$, then \eqref{eq:RZero} implies after canceling the telescoping factors $(1-q^c)$ for $c\neq w_i^R,w_i^L$ that
\[
A^{\lambda,w}_x(q,1) =   \prod_{i=1}^m \frac{1 -q^{w^{L}_i}}{1 -q^{w^R_i+1}} .
\]
Dividing numerator and denominator by $(1-q)^m$ yields \eqref{eq:RZeroLem}.

When $r>0$, there must exist a $j$ such that $w_j^R$ is on the left of $x$. Taking $j$ to be the maximal such index, then \eqref{eq:RPositive} implies, after splitting the product into two products and canceling telescoping factors, that
\[
A^{\la,w}_x(q,1)= 
q^{  - 1}   \frac{q-1}{q^r -1} \frac{1}{q^r-q^{w_j^R}}\frac{\prod_{i=1}^j q^r  - q^{w^L_i}}{\prod_{i=1}^{j-1} q^r  -q^{w^{R}_i +1}}
\prod_{i=j+1}^m \frac{q^r  - q^{w^L_i}}{q^r  -q^{w^{R}_i +1}}.
\]
Dividing the numerators and denominators of each factor by $1-q$ and factoring out $q$ powers yields \eqref{eq:RPositiveLem}. Note that when $i=1$ in the denominator of the first product, $[r-w_{0}^R-1]_q = [r]_q$, which is inherited from moving $(\frac{q-1}{q^r-1})$ into the denominator.
\end{proof}

For each $T$ we have 
\[
A^\e_T(q,1) = \prod_{i=1}^n A_{T^{(i)}}^{\shape(T_{<i}),w(T_{\prec i})}(q,1).
\]
 We evaluate each factor $A_{T^{(i)}}^{\shape(T_{<i}),w(T_{\prec i})}(q,1)$ by Lemma~\ref{lem:Aq1} and apply Proposition~\ref{prop:FDExpansion}. 
The total $q$ power coefficient in the product $q^{\binom{n}{2}- |\e|} A_T^\e(q,1)$ is the following, where $f_T(q)$ is the product over all $q^{b(w,x)}$ factors that appear from \eqref{eq:RPositiveLem} for each label $i$ such that $x=T^{(i)}$ is not in the first column:

\begin{gather*}
q^{\binom{n}{2} - |\e|} ~f_T(q)~ \prod_{(a,b)\in \la} q^{-a} 
= q^{\binom{n}{2} - |\e|- \sum_i \binom{\la_i}{2}}~f_T(q).
\end{gather*}
From the fact that
\[
\sum_{i<j} \lambda_i \lambda_j  = \frac{1}{2} 
\left( 
\sum_i \lambda_i \sum_j \lambda_j - \sum_i \lambda_i^2
\right) 
= \binom{n}{2} - \sum_i \binom{\la_i}{2},
\]
we may simplify the $q$ power above to
\[
\binom{n}{2}-|\e| - \sum_i \binom{\la_i}{2}  = \sum_{i<j}\la_i \la_j - \sum_i \e(i).
\]
Finally, observe that $(1-q)^{-n} (q;q)_\la = \prod_i [\la_i]_q!$. 

For $x,\la,w$ as above, define
\[
\hat A_x^{\la,w} \coloneqq q^{b(w,x)} \prod_{i=1}^j \frac{[r-w_i^L]_q}{[r-w_{i-1}^R-1]_q}\prod_{i=j+1}^m \frac{[w_i^L-r]_q}{[w_i^R+1-r]_q},
\]
where $j\coloneqq 0$ in the case when $r=0$, and $j$ is the largest such that $w_j^R$ is to the left of $x$ when $r>0$; 
set
\[
\hat A_T^\e \coloneqq \prod_{i=1}^n \hat A_{T^{(i)}}^{\shape(T_{<i}),w(T_{\prec i})},
\]
which gives $f_T(q)A_T^\e(q,1)$. Then by Proposition \ref{prop:FDExpansion} the coefficient of $e_\la$ in $\chi_\e[X;q]$ is
\[
q^{\sum_{i<j}\la_i \la_j - \sum_i \e(i)}\prod_i[\la_i]_q!\sum_{T\in \overline{\SYT}^\e_\la}  \hat A_T^\e,
\]
which matches Hikita's formula~\cite[Theorem 3]{Hikita}.

\subsection{Proof of the main theorem}
We end the section by consolidating the proof of our main theorem, Theorem~\ref{thm:IntroMainThm}.

\begin{proof}[Proof of Theorem~\ref{thm:IntroMainThm}]
By Proposition~\ref{prop:FDExpansion}, 
we have expansions of $F_\e[X;q]$ and $\chi_\e[X;q]$ in terms of $\Ht_\lambda[X;q,t]$ and $\Ht_\lambda[(q-1)X;q,t]$, respectively. 
By Corollary~\ref{cor:main_corollary}, we have an explicit expansion of $\chi_\e$ in terms of the elementary basis, Equation~\eqref{eq:FirstPositiveExpansion}, and an explicit expansion at $q=1$, Equation~\eqref{eq:FirstPositiveExpansion1}, proving the Stanley--Stembridge conjecture.
Lemma~\ref{lem:Aq1} shows that our formula is equivalent to that of Hikita's.
\end{proof}

\section{A Hall--Littlewood expansion}

In this section, we use Theorem~\ref{thm:IntroMainThm} to give an expansion of $\chi_\e$ in terms of Hall--Littlewood symmetric functions.
Combining Equation~\eqref{eq:HtoJ} with
\cite[VI.8.4.ii]{Macdonald-book}, we find that
\[
\Ht_\lambda[(q-1)X;q,0] 
= q^{|\lambda'| + n(\lambda')} J_{\lambda'}[X;0,q^{-1}]
=
q^{|\lambda'| + n(\lambda')} Q_{\lambda'}[X;q^{-1}].
\]

Recall from Proposition~\ref{prop:FDExpansion} that
\[
\chi_\e[X;q] = (q-1)^{-n}\sum_\la  ~q^{\binom{n}{2} - |\e|}~
\left(\sum_{T \in \SYT^\e_\lambda} A_T^{\e}(q,t)\right) \Ht_\la[(q-1)X;q,t].
\]
Next, we specialize $A_x^{\la,w}$ at $t=0$. Let $\la,w,x$ be as in Proposition~\ref{prop:FDExpansion}, and recall that
\begin{equation}
 A^{\lambda,w}_x \coloneqq 
  \prod_{c\in \mathcal{R}_{\lambda,x}} \frac{ q^{a_\lambda(c)} - t^{l_{\lambda}(c)+1}}{q^{a_\lambda(c)+1} - t^{l_\lambda(c)+1}}
\,\times\,\prod_{c\in \mathcal{C}_{\lambda,x}} \frac{ q^{a_\lambda(c)+1} - t^{l_{\lambda}(c)}}{q^{a_\lambda(c)+1} - t^{l_\lambda(c)+1}}
\,\times\, \prod_{i=1}^k \frac{x-tw_i}{x-qtw_i}.\label{eq:Axlamw-HL}
\end{equation}
Letting $x=q^r t^s$, then we have two cases: either $s=0$ or $s>0$. In the first case when $s=0$, the product over $\mathcal{C}_{\la,x}$ is empty, and at $t= 0$, the other factors of \eqref{eq:Axlamw-HL} become
\[
\prod_{c\in \mathcal{R}_{\la,x}} \frac{q^{a_\la(c)}}{q^{a_\la(c)+1}}\times \left.\left(\prod_{i=1}^k\frac{q^r-tw_i}{q^r-qtw_i}\right)\right|_{t=0} = q^{-r}.
\]

In the second case when $s>0$, then all factors in the product over $\mathcal{C}_{\la,x}$ in \eqref{eq:Axlamw-HL} cancel except for when $c$ is the cell $x'$ directly underneath $x$, in which case $l_\la(x') = 0$. Since $a_\la(x') = \la_s-r-1$, then $a_\la(x')+1 = \la_s-r = \la_s-\la_{s+1}$. Then, at $t=0$, the expression $A^{\la,w}_x$ becomes
\[
 \prod_{c\in \mathcal{R}_{\la,x}} \frac{q^{a_\la(c)}}{q^{a_\la(c)+1}} \times \frac{q^{\la_s-\la_{s+1}} - 1}{q^{\la_s-\la_{s+1}}} \times \left.\left(\prod_{i=1}^k\frac{x-tw_i}{x-qtw_i}\right)\right|_{t=0}.
\]
There are three cases for each factor in the last product: Let $w_i = q^{r_i}t^{s_i}$. First, if $s\leq s_i$, then
\[
\left.\frac{x-tw_i}{x-qtw_i}\right|_{t=0} = \left.\frac{1-tw_i/x}{1-qtw_i/x}\right|_{t=0} = 1.
\]
Second, if $s>s_i+1$ then
\[
\left.\frac{x/(tw_i)-1}{x/(tw_i)-q}\right|_{t=0} = q^{-1}.
\]
Third, if $s=s_i+1$ then
\[
\frac{1-tw_i/x}{1-qtw_i/x}\Big|_{t=0} = \frac{1-q^{r_i-r}}{1-q^{r_i-r+1}}.
\]
Thus,
\begin{align*}
A^{\la,w}_x(q,0) &= q^{-r-m}\frac{q^{\la_s-\la_{s+1}}-1}{q^{\la_s - \la_{s+1}}} \times \crap\\
&\hspace{3em} \text{ with } \crap = \prod_{i: s_i = s-1} \frac{q^{r_i-r}-1}{q^{r_i-r+1}-1},
\end{align*}
where $m$ is the number of $w_i$ which are in a strictly lower row than $x$, and the factor $\crap$ is $1$ if there exists no $i$ such that $s_i=s-1$; and otherwise
\[
\crap = \frac{q^{w^L-r}-1}{q^{\la_{s}-r}-1},
\]
where $w^L$ is the column (starting from $0$) of the left-most $w_i$ with $s_i=s-1$.

Now, given $T\in \SYT^\e_\la$ and a label $i$, let $\low(T,i)$ be the number of entries of $T_{\prec i}$ that are at least 2 rows below $T^{(i)}$. Furthermore, if $T^{(i)}$ is in row $s+1$ (starting from $1$), then let $\rowdiff(T,i) \coloneqq \shape(T_{< i})_s - \shape(T_{< i})_{s+1}$. If $T_{\prec i}$ contains elements in row $s$, let $\L(T,i)$ be the column of its left-most element in row $s$ of $T_{\prec i}$. Finally, we write $T_1$ for the first row of $T$.

Observe that the products over all $q^{-r}$ as $i$ runs over entries of some $T$ is $q^{-n(\la')}$. Thus, by the above discussion we have
\begin{align*}
\chi_\e[X;q] = (q-1)^{-n}\sum_\la Q_{\la'}&[X;q^{-1}] q^{ \binom{n+1}{2} - |\e|} \\ & \times \left(\sum_{T \in \SYT^\e_\lambda}\prod_{i \notin T_1} q^{- \low(T,i)} \frac{q^{\rowdiff(T,i)}-1}{q^{\rowdiff(T,i)}}\prod_{i:(T_{\prec i})_s\neq \emptyset} \frac{q^{\L(T,i)-\shape(T_{<i})_{s+1}}-1}{q^{\rowdiff(T,i)}-1}\right).
\end{align*}
We may rewrite the factor in parentheses as
\[
\sum_{T \in \SYT^\e_\lambda} \prod_{i \notin T_1} q^{-\low(T,i)-\rowdiff(T,i)}  \prod_{\substack{i \notin T_1,\\(T_{\prec i})_{s}=\emptyset}} (q^{\rowdiff(T,i)}-1) \prod_{i:(T_{\prec i})_s\neq \emptyset} (q^{\L(T,i)-\shape(T_{<i})_{s+1}}-1),
\]
and by dividing by a power of $q-1$ we may rewrite each factor as a $q$-number. 
Therefore, we have the following theorem.

\begin{thm}
We have the following expansion of $\chi_\e$ in terms of reversed $Q_\la$ symmetric functions:
\begin{align*}
\chi_\e[X;q] = \sum_\la Q_{\la'}&[X;q^{-1}] \frac{ q^{ \binom{n+1}{2} - |\e|} }{(q-1)^{\la_1}}
\\
&\times \left(\sum_{T \in \SYT^\e_\lambda} \prod_{i \notin T_1} q^{-\low(T,i)-\rowdiff(T,i)}  \prod_{\substack{i \notin T_1,\\(T_{\prec i})_{s}=\emptyset}} [\rowdiff(T,i)]_q \prod_{i:(T_{\prec i})_s\neq \emptyset} [\L(T,i)-\shape(T_{<i})_{s+1}]_q\right).
\end{align*}
\end{thm}

In \cite{Aggarwal-Borodin-Wheeler}, an expansion for arbitrary LLT polynomials is given in terms of the Hall-Littlewood basis. It would be interesting to see how our formula is related to theirs in the special case of unicellular LLTs.

\section{Acknowledgments}
This work was supported by ERC grant ``Refined invariants in combinatorics, low-dimensional topology and geometry of moduli spaces,'' No. 101001159.

\bibliographystyle{plainurl}
\bibliography{Chromatic.bib}

\end{document}